\documentclass[12pt,leqno]{article}

\usepackage{palatino,mathrsfs}
\usepackage[mathcal]{euler}

\usepackage{xcolor}

\usepackage{graphicx} 
\usepackage{epstopdf,epsfig}

\usepackage{comment} 
\usepackage{amsmath,amsthm,amsfonts,amssymb,latexsym,amscd,enumerate,url,hyperref}
\usepackage{amssymb}
\usepackage{url}
\usepackage{graphicx}
\usepackage[all,knot]{xy}
\xyoption{arc}

\usepackage{chngcntr}
 \counterwithin*{equation}{section}
  \counterwithin{equation}{section}

\usepackage{multirow}

\theoremstyle{plain}
\newtheorem{theorem}{Theorem}[section] 
 
\newtheorem{lemma}[theorem]{Lemma} 
\newtheorem{proposition}[theorem]{Proposition} 

\newtheorem*{theorem-non}{Theorem}

\theoremstyle{definition}
\newtheorem{defn}[theorem]{Definition} 
\newtheorem{example}[theorem]{Example}

\newtheorem{remark}[theorem]{Remark} 

\def\R{\mathbb R}
\def\C{\mathbb C}

\def\P{\mathbb P}

\def\R{\mathbb R}
\def\C{\mathbb C}

\def\P{\mathbb P}

\def\BG{\boldsymbol{G}}
\def\g{\boldsymbol{\mathfrak{g}}}

\def\fg{\mathfrak{g}}

\usepackage{comment} 

\begin{document}

\title{Algebraic Families of  Groups and Commuting Involutions}
%
\author{Dan Barbasch\thanks{Department of Mathematics, Cornell University, Ithaca, NY 14850, USA.} , \ Nigel Higson\thanks{Department of Mathematics, Penn State University, University Park, PA 16802, USA.} \ and Eyal Subag\footnotemark[2]}

\date{}

\maketitle


\begin{abstract}
\noindent Let $G$ be a complex affine algebraic group, and let $\sigma_1$ and $\sigma_2$ be commuting anti-holomorphic involutions of $G$. 
We construct  an algebraic family of algebraic groups  over the complex projective line and a real structure on the family   that interpolates between the real forms $G^{\sigma_1}$ and $G^{\sigma_2}$.
 \end{abstract}

\section{Introduction}

The purpose of  this note is to construct  real forms of one-parameter  
  algebraic  families  of   complex affine algebraic
  groups.   The notion of an algebraic
family is as in \cite{Ber2016,Ber2017}, and our results generalize the
examples introduced in \cite[Sec. 3]{Ber2016}.  The examples of most interest concern  complex semisimple  groups.  Our construction  shows that one can smoothly, and indeed   algebraically, interpolate between various real forms, whether or not they are in the same inner class.

Let  $G$ be a  complex affine algebraic group.  If 
$\sigma$ 
is an antiholomorphic involution of $G$, then we shall define, as usual, $G^\sigma$ to be the fixed subgroup
\[
G^\sigma = \{\, g\in G : \sigma (g) = g\,\} ,
\]
which is a closed Lie subgroup of $G$.
Passing to Lie algebras,  define the real Lie subalgebra $\fg^{\sigma}\subseteq \fg$ similarly, and define
\[
\fg^{-\sigma} = \{\, X \in \fg : \sigma(X)  = -X\,\}.
\]
We shall prove the following result:

\begin{theorem-non}[\textbf{\ref{thm-two-real-forms}}] 
Let $G$ be a complex affine algebraic group and let $\sigma_1$ and $ \sigma_2$ be two commuting antiholomorphic involutions of  $G$.
  There exists an algebraic family  $\boldsymbol G$ of  affine algebraic groups    over $\C\P^1$, and an antiholomorphic involution $\boldsymbol{\sigma}$ of the family $\boldsymbol G$    that is compatible with the standard real structure on $\C\P^1$, 
  such that if   $[\alpha:\beta]\in \R\P^1$, then 
\begin{eqnarray}\nonumber
&&{\boldsymbol G}^{\boldsymbol{\sigma}}|_{[\alpha:\beta]}\cong \begin{cases}
G^{\sigma_1}  & \alpha\beta> 0\\
   \left ( G^{\sigma_1}\cap G^{\sigma_2}\right )   \ltimes   \left ( \fg^{\sigma_1} \cap \fg^{-\sigma_2}  \right )  &  \alpha\beta=0 \\
G^{\sigma_2}  & \alpha\beta< 0 .
\end{cases}
\end{eqnarray}
\end{theorem-non}

See \cite[Section 2]{Ber2016}  and  Section~\ref{sec-group-families} below for the concepts of algebraic family and antiholomorphic involution that we are using here.  

Since  the involutions $\sigma_1$ and $\sigma_2$ in the theorem commute,   $\sigma_2$ restricts to an involution of  $\fg^{\sigma_1}$ and   $\sigma_1$ restricts to an involution of  $\fg^{\sigma_2}$.  Hence 
\[
\begin{aligned}
 \fg^{\sigma_1} & = \left (  \fg^{\sigma_1} \cap \fg^{\sigma_2} \right ) \oplus  
\left (  \fg^{\sigma_1} \cap \fg^{-\sigma_2} \right ) \\
 \fg^{\sigma_2} & = \left (  \fg^{\sigma_1} \cap \fg^{\sigma_2} \right ) \oplus 
\left (  \fg^{-\sigma_1} \cap \fg^{\sigma_2} \right ) .\\
\end{aligned}
\]
Moreover  
\begin{equation*}
\sqrt{-1}\cdot 
\left (  \fg^{\sigma_1} \cap \fg^{-\sigma_2} \right ) = 
\left (  \fg^{-\sigma_2} \cap \fg^{\sigma_2} \right ) .
\end{equation*}
In  the context of Riemannian symmetric spaces the algebras $\fg^{\sigma_1}$ and $\fg^{\sigma_2}$ are  called  \emph{dual real forms} of $\fg$; see  \cite[Section V.2]{Hel}. A  more general concept
  of symmetric space, appropriate to the study of involutions on general groups, is considered in   \cite{Loos}.

It is   easy to find instances where the hypotheses of Theorem~\ref{thm-two-real-forms} are satisfied.  For instance a   theorem of E.\ Cartan asserts that any   antiholomorphic involution  of a connected complex semisimple group commutes with an antiholomorphic involution  whose fixed group is a maximal compact subgroup (see  for example   \cite[Theorem 5.1.3]{OnishchikVinberg}).  Other examples are the antiholomorphic involutions defining the real forms $SU(p,n{-}p)$ of $SL(n,\C)$.

The relation between complex-linear involutions and conjugate-linear involutions of a semisimple complex Lie
  algebra is well known, and the classification of either  is also well known. The question of when two
conjugacy classes of involutions have representatives that commute is studied in \cite{HelmSchw} and \cite{Helm} (the answer depends on the precise notion of conjugacy that one uses).  In the case of groups, as opposed to Lie algebras, the isogeny class also plays a role. See for example \cite[Example 16.4]{VoganDuality} for the case of $D_4$, and see \cite{ABV} for a discussion of additional subtleties in the definition of a real form that are of importance in representation theory.

Theorem  \ref{thm-two-real-forms}   extends to all group cases a construction made  at the Lie algebra level in   \cite{Ber2016} and studied in some group cases there.  The focus of \cite{Ber2016} was on the case where $G^{\sigma_2}$ was a compact form of $G$.  In that case $G^{\sigma_1} \cap G^{\sigma_2}$ is  a maximal compact subgroup of $G^{\sigma_1}$ and it is natural to study algebraic families of Harish-Chandra modules associated to the algebraic families of groups given by Theorem~\ref{thm-two-real-forms}. A start on this was made in  \cite{Ber2016,Ber2017,Subag17}, where connections were made to the  ``Mackey bijection'' between irreducible representations of $G^{\sigma_1}$ and those of its Cartan motion group \cite{Mackey75,Higson08,Higson11,Afgoustidis15}, as well as to   contraction families of representations studied in mathematical physics \cite{Inonu-Wigner53,Dooley83,Dooley85}.

The cases where $G^{\sigma_1} \cap G^{\sigma_2}$ is not  maximal compact present interesting new challenges.
In representation theory     one is often lead to consider at the same time different real forms of a single complex algebraic  group. This occurs  in the context   Vogan's duality; see  e.g., \cite[Conjecture 4.15]{Vogan93}, \cite[Theorem 1.18]{ABV}, and  \cite{VoganDuality}. More recently  Bernstein has voiced the idea of  grouping together various real forms in order to define the  ``correct'' category of representations for real groups; see   \cite{Bernstein14}.  See also the work of Flensted-Jensen \cite{Flensted-Jensen}, where  spherical functions on dual  real forms are   related to one another.   It remains an   open problem to develop a concept of families of modules  that might be   appropriate and useful to these situations.  Perhaps the concept of Schwartz space from \cite{AizenbudGourevitch08} and the associated idea of SF-module \cite{BernsteinKroetz2014} might be relevant here.

We note that  degenerations of symmetric pairs  over algebraically closed fields and related matters   were considered before e.g., see \cite{Panyushev} and references therein. 

The   authors thank Jeffrey Adams and Joseph Bernstein for very helpful discussions.    Dan Barbasch was supported by an NSA grant.

\section{An algebraic family of   Lie algebras}
\label{sec-algebra-families}
 An  \emph{algebraic family   of complex Lie algebras} over a  complex variety $X$ is  a locally free sheaf $\g$  of $O_{X}$-modules that is equipped  with $O_{X}$-linear Lie bracket  
 \[
 [\,\_\,,\_\,]:\g \times_{O_X} \g \longrightarrow \g .
 \]
See   \cite{Ber2016}.    Now let   $\fg $ be  a finite-dimensional complex   Lie algebra and let 
 \[
 \theta \colon \fg \longrightarrow \fg
 \]
be a complex-linear  involution of $\fg$.  In \cite{Ber2016} a nontrivial  algebraic family of complex Lie algebras was constructed using the pair $(\fg, \theta)$. The purpose of this section is to give a second version of the construction that   will be well suited to a later adaptation   from Lie algebras to groups.

 To start, fix a complex Lie algebra embedding  
\[
\iota:\fg\hookrightarrow  \mathfrak{sl}(n,\C).
\]  
Then the formula 
\begin{equation}
\label{eq-lie-alg-2n-embedding}
X \longmapsto \frac 12 
\begin{bmatrix} 
\iota(X) + \iota(\theta (X)) & \iota(X) - \iota(\theta (X)) \\
\iota(X) - \iota(\theta (X)) & \iota(X) + \iota(\theta (X)) 
\end{bmatrix}
\end{equation}
defines a  Lie algebra embedding of $\fg$ into $\mathfrak{sl}(2n,\C)$. The image is the Lie subalgebra
\[
\fg_{[1:1]}  = \left\{ \,  \begin{bmatrix}
X &  Y \\
  Y&X
\end{bmatrix} \in \mathfrak{sl}(2n,\C) \,:\,  X\in \iota(\fg^{\theta} ),Y\in \iota(\fg^{-\theta} ) 
\,\right\} .
 \]
Now if  $[\alpha {:} \beta ]$ is a point of $\C\P^1$ (with homogeneous coordinates $\alpha$ and $\beta$), then the complex vector space 
\[
\fg_{[\alpha:\beta]} = \left\{ \,  \begin{bmatrix}
X &  \alpha Y \\
  \beta Y&X
\end{bmatrix} \in \mathfrak{sl}(2n,\C) \,:\,  X\in \iota(\fg^{\theta} ),Y\in \iota(\fg^{-\theta} ) 
\,\right\}
 \]
 is a Lie subalgebra of $\mathfrak{sl}(2n,\C)$ that depends only on the point $[\alpha{:}\beta]$. 
 
 The subalgebras $\fg_{[\alpha:\beta]} $ are the fibers of an algebraic vector subbundle of the trivial bundle over $\C\P^1$ with fiber $\mathfrak{sl}(2n,\C)$.   The sheaf of sections $\g$ of this algebraic vector bundle is an algebraic family of Lie algebras over $\C\P^1$ with fibers $\g |_{[\alpha:\beta]} = \fg_{[\alpha:\beta]}$.   
 
 In \cite{Ber2016} the same family was constructed in a different way, as follows.  Start with the decomposition
\begin{equation}
\label{eq-family-cartan-decomp}
O_{\C} \otimes_{\C} \fg = O_{\C} \!\otimes_{\C}\! \fg^{\theta} \, \oplus \, O_{\C}\! \otimes_{\C} \!\fg^{-\theta}
\end{equation}
and  define a Lie bracket operation on sections belonging to   individual summands by 
\[
[\eta, \zeta](z) =  \begin{cases}
z \cdot [\eta(z), \zeta(z)]_{\mathfrak g}   & \text{if $\zeta$  and $\eta$ are sections of $ O_X\! \otimes_{\C} \!\fg^{-\theta}$}   \\
\phantom{i} [\eta(z), \zeta(z)]_{\mathfrak g}  & \text{otherwise}.
\end{cases}
\]
We obtain an algebraic family of Lie algebras over $\C$.

Setting $z= \alpha/\beta$, we obtain an algebraic family of Lie algebras over  the complement of $[1{:}0]$ in $\C\P^1$, while setting $z = \beta / \alpha$ we obtain   an algebraic family over the complement of $[0{:}1]$ in $\mathbb{CP} ^1$.   The formula 
\[
\eta  \longmapsto 
\begin{cases}
[\alpha{:}\beta]\mapsto 
\frac {\alpha}{\beta} \cdot \eta(\frac {\alpha}{\beta})  & \text{if   $\eta $ is a section of $ O_X\! \otimes_{\C}\, \!\fg^{-\theta}$}   \\
[\alpha{:}\beta]\mapsto \phantom{i} \eta(\frac {\alpha}{\beta})  & \text{if   $\eta $ is a section of $ O_X\! \otimes_{\C}\, \!\fg^{\theta}$}
\end{cases}
\]
defines an isomorphism from the family over $\mathbb{CP} ^1 \setminus \{ [1{:}0]\}$ to the family over   $\mathbb{CP} ^1 \setminus \{ [0{:}1]\}$, on the overlap where both families are defined, and with this identification we obtain the algebraic family of Lie algebras  on $\C\P^1$ from \cite{Ber2016}.

The formula 
\[
\eta \longmapsto
\begin{cases}
[\alpha{:}\beta] \mapsto \left [ \begin{smallmatrix}
\eta (\frac \alpha \beta) &  0 \\ 0 & \eta (\frac \alpha \beta)
\end{smallmatrix}\right ]_{\vphantom{\int}}
 & \text{if  $\eta\in O_\C \otimes_\C \fg^{\theta}$ } \\
[\alpha{:}\beta] \mapsto  \left [ \begin{smallmatrix}
0&\frac \alpha \beta    \eta (\frac \alpha \beta)   \\   \eta (\frac \alpha \beta) & 0
\end{smallmatrix}\right ]
 &\text{if  $\eta\in O_\C \otimes_\C \fg^{-\theta}$  } \\
\end{cases}
\]
defines an isomorphism over $\C\P^1 \setminus \{ [1{:} 0]\}$  from the family of \cite{Ber2016} to the family we have constructed in this section that extends to an isomorphism over all of $\C\P^1$ by a similar formula.

Of course this identification shows that the construction of $\g$  in this section is independent of the choice of embedding of $\fg$  into $\mathfrak{sl}(n,\C)$.  From either point of view, it  is easy to see that 
\[
\g\vert _{[\alpha:\beta]}\cong 
\begin{cases}
\fg &  \alpha\beta\neq 0\\
\fg^{\theta} \ltimes \fg^{-\theta}&  \alpha\beta=0  .
\end{cases}
\]

Conjugation by the matrix $\left [\begin{smallmatrix} -I & 0 \\ 0 & I\end{smallmatrix}\right ] $
defines an involution $\boldsymbol \theta$ of the family   $\g $ constructed in this section, and   we obtain  a decomposition 
\[
\g=\g^{\boldsymbol{\theta}} \oplus \g ^{-\boldsymbol{\theta}}
\]
as $O_{\C\P^1}$-modules (over $\C\subseteq \C\P^1$ this is of course the decomposition \eqref{eq-family-cartan-decomp}). 
The family $\g$ is nontrivial even as a sheaf of $O_{\C\P^1}$-modules, since the sheaf  $\g^{\boldsymbol{\theta}}$ is isomorphic  to $O_{\C\P^1}\otimes_{\C} \fg^{\theta}$, but     $\g^{-\theta}$ is isomorphic to $O_{\C\P^1}(-1)\otimes_{\C}\fg^{-\theta}$.

\section{Real structure on the family of Lie algebras }
\label{sec-real-structure-algebras}
 
 If $X$ is a complex algebraic variety, then as usual we shall denote by $\overline X$ the complex conjugate variety. It has the same underlying topological space  as $X$, but  a complex-valued function defined on some open set of $\overline X$  is regular if and only if the complex conjugate function is regular on $X$.  A morphism $X\to \overline Y$ can also be viewed as a morphism $\overline X \to Y$.  So we can speak of an involutive morphism $\sigma\colon X\to \overline X$, and this is by definition a \emph{real structure} on $X$, or (conventionally but a little misleadingly) an \emph{antiholomorphic involution} of $X$.

To each sheaf $S$ of $O_X$-modules there is an obvious associated \emph{complex conjugate} sheaf $\overline S$ of $O_{\overline X}$-modules. Given a real structure $\sigma$ on $X$, a  \emph{real structure} on  the sheaf is an involutive  morphism $S\to \sigma^* \overline S$.  This concept specializes to algebraic families of complex Lie algebras in the obvious way \cite[Sec. 2.5]{Ber2016}.

Let $\g$ be an algebraic family of Lie algebras over $X$, equipped with a real structure.  
We shall be interested in the fixed set $X^\sigma\subseteq X$ (in general it could be empty) and in  the fibers of the fixed sheaf $\g^\sigma$ over $X^\sigma$.  If $X$ is nonsingular, then $X^\sigma$ is a smooth submanifold of $X$, and the fibers of $\g^\sigma$ constitute the fibers of a smooth vector bundle over $X^\sigma$, with smoothly varying Lie algebra structures.

We shall be concerned throughout the note with   the \emph{standard} real structure on  $X=\C\P^1$ given by  $ [\alpha{:}\beta]\mapsto [\overline{\alpha}{:}\overline{\beta}]$.   Our aim is to construct a real structure on the family $\g$ constructed in the previous section, and to prove the following result:

 \begin{theorem}
 \label{thm-real-fmly-algebras1}
Let $\fg$ be a finite-dimensional complex Lie algebra, let $\theta$ be a complex-linear involution of $\fg$, and let $\g$ be the algebraic family of Lie algebras over $\C\P^1$ from the previous section associated to $\fg$ and $\theta$.  If $\sigma $ is any conjugate-linear involution of $\fg$ that commutes with $\theta$, then there is an associated  real structure $\boldsymbol \sigma$  on $\g$  whose  fixed fibers over the fixed space $\R\P^1$ for the standard real structure on $\C\P^1$ have the following isomorphism types:
\begin{eqnarray}\nonumber
&&\g ^{\sigma}|_{[\alpha:\beta]}\cong 
 \begin{cases}
\fg^{\sigma} & \alpha\beta> 0\\
 \left (\fg^{\sigma}\cap \fg^{\theta}\right ) \ltimes \left ( \fg^{\sigma}\cap\fg^{-\theta}\right )&  \alpha\beta=0 \\
 {\fg}^{\sigma\theta}  & \alpha\beta< 0 .
\end{cases}
\end{eqnarray}
\end{theorem}

To that end, fix a conjugate-linear involution 
$\sigma$ 
of the finite-dimensional Lie algebra $\fg$, and assume  that  $\sigma$ {commutes} with the given complex-linear involution $\theta$.  
 
 Given $\sigma$, it  is very easy to define a real structure on $\g$ using the description of our family from \cite{Ber2016}: just use  the formula
 \[
 {\boldsymbol\sigma}  (\eta)\left (\strut [\alpha{:}\beta]\right ) = \sigma \left ( \eta( [\overline \alpha{:}\overline \beta])\right ) ,
 \]
 where $\eta$ is a section of the sheaf \eqref{eq-family-cartan-decomp}.
 But it will be convenient to translate this into a formula that uses the matrix construction of $\g$.  For this purpose it is helpful (but not necessary) to choose the initial embedding 
 \[
 \iota \colon \fg\longrightarrow \mathfrak{sl}(n,\C)
 \]
appropriately.

\begin{lemma}\label{lem-can-form-comm-inv1}
There exists an embedding $\iota $  as above  and a matrix $S\in GL(n,\R)$ such that  $S^2 = I$, and such that the conjugate linear involution 
\[
\tau \colon X \longmapsto  S\overline{X}S 
\]
 on $\mathfrak{sl}(n,\C)$ \textup{(}the overline denotes entrywise complex conjugation\textup{)} satisfies 
 \[
 \iota (\sigma(X)) =  \tau(\iota(X))
 \]
  for every $X \in \mathfrak {g}$.\end{lemma}

\begin{proof} 
Start with any embedding of $\mathfrak{g}$ into some $\mathfrak{sl}(k,\C)$.  Then set $n=2k$ and embed $\mathfrak g $ into $\mathfrak{sl}(n,\C)$ via the map
\[
X \mapsto 
\begin{bmatrix}
X & 0 \\ 0 & \overline{\sigma(X)}
\end{bmatrix} 
\]
The matrix
\[
S = \begin{bmatrix} 0& I \\ I & 0 \end{bmatrix} \]
has the required property.
\end{proof}

\begin{remark}
The matrix $S$ above is conjugate to $\left [ \begin{smallmatrix} I & 0 \\ 0 & -I \end{smallmatrix} \right ]$ in $GL(n,\R)$.  So the fixed subalgebra of the conjugate-linear involution $\tau$ above is conjugate to $\mathfrak {su} (k,k)\subseteq \mathfrak{sl}(2k,\C)$. 
\end{remark}

Fix an embedding and matrix $S$ as in the lemma.  The formula 
\begin{equation}
\label{eq-2by2-involution}
 \begin{bmatrix} X & Y \\ Z & W 
\end{bmatrix}
\longmapsto
\begin{bmatrix}
S & 0 \\ 0 & S 
\end{bmatrix}
\begin{bmatrix}
\overline X & \overline  Y \\ \overline  Z & \overline  W
\end{bmatrix}
\begin{bmatrix}
S  & 0 \\ 0 & S 
\end{bmatrix}
\end{equation}
defines a conjugate-linear involution of $\mathfrak{sl}(2n,\C)$ that again extends the original conjugate-linear involution on $\fg$, this time under the embedding \eqref{eq-lie-alg-2n-embedding}.  The involution \eqref{eq-2by2-involution} defines a real structure  in the obvious way    on the constant family of Lie algebras over $\C\P^1$ with fiber $\mathfrak{sl}(2n,\C)$. This restricts to a real structure $\boldsymbol \sigma$ on $\g$ since  the involution  \eqref{eq-2by2-involution} maps 
$
\fg _{[\alpha{:}\beta]}$ to $ \fg _{[\overline \alpha{:}\overline\beta]}
$, for all $[\alpha{:}\beta]\in \C\P^1$.  Indeed the action of \eqref{eq-2by2-involution}  on $\fg_{[\alpha{:}\beta]}$ is 
\begin{equation*}
\label{eq-upsilon}
\begin{bmatrix} \iota(X) & \alpha \iota(Y) \\ \beta \iota(Y)  & \iota(X)
\end{bmatrix}
\longmapsto
\begin{bmatrix} \iota(\sigma(X)) & \overline \alpha \iota(\sigma(Y)) \\  \overline \beta  \iota(\sigma(Y)) & \iota(\sigma(X)) 
\end{bmatrix}
\end{equation*}
for all $X\in  \fg^{\theta}$ and all  $Y\in \fg^{-\theta}$.

Suppose now that $[\alpha{:}\beta]\in \R\P^1$, so that we get a conjugate-linear involution
\[
\sigma_{[\alpha{:}\beta]} \colon  
\fg _{[\alpha{:}\beta]}\longrightarrow   \fg _{[  \alpha{:} \beta]} .
\]
  The fixed-point algebra is 
\[
\fg_{[\alpha{:}\beta]}^\sigma
= 
\left\{ \,  
\begin{bmatrix}
X &  \alpha Y \\
  \beta Y&X
\end{bmatrix} \in \mathfrak{gl}_{2n}(\C) 
\,:\,  
X\in \iota(\fg^{\theta} \cap \fg ^{\sigma}),Y\in \iota(\fg^{-\theta} \cap \fg^{\sigma}) 
\,\right\} .
\]

When $\alpha/\beta > 0$,  upon conjugating by $\left [ \begin{smallmatrix} \gamma I  & 0 \\ 0 & I\end{smallmatrix} \right ]$, where $\gamma$ is a  square root of $ \beta/\alpha$, we find that 
\begin{equation}
\label{eq-alpha-beta-pos}
\fg_{[\alpha{:}\beta]}^\sigma
\cong 
\left\{ \,  
\begin{bmatrix}
X &    Y \\
    Y&X
\end{bmatrix} \in \mathfrak{gl}_{2n}(\C) 
\,:\,  
X\in \iota(\fg^{\theta} \cap \fg ^{\sigma}),Y\in \iota(\fg^{-\theta} \cap \fg^{\sigma}) 
\,\right\} ,
\end{equation}
and the right-hand side of \eqref{eq-alpha-beta-pos} is isomorphic to $\fg^{\sigma}$ via the embedding \eqref{eq-lie-alg-2n-embedding}.

When $\alpha/\beta < 0$,  upon conjugating by $\left [ \begin{smallmatrix} i\gamma I & 0 \\ 0 & I \end{smallmatrix} \right ]$, where $\gamma$ is a   square root of $ - \beta/\alpha$, we find that 
\begin{equation}
\label{eq-alpha-beta-neg}
\fg_{[\alpha{:}\beta]}^\sigma
\cong 
\left\{ \,  
\begin{bmatrix}
X &   i Y \\
   i Y&X
\end{bmatrix} \in \mathfrak{gl}_{2n}(\C) 
\,:\,  
X\in \iota(\fg^{\theta} \cap \fg ^{\sigma}),Y\in \iota(\fg^{-\theta} \cap \fg^{\sigma}) 
\,\right\} .
\end{equation}
Note now that 
\[
i\cdot \left ( \fg^{-\theta} \cap \fg^{\sigma}\right ) 
=
 \left ( \fg^{-\theta} \cap \fg^{-\sigma}\right )  .
\]
From this we find that the right-hand side of \eqref{eq-alpha-beta-neg}   is isomorphic to 
\[
\fg^{\theta\sigma} = \left (\fg^{\theta} \cap \fg ^{\sigma}\right ) \oplus  \left ( \fg^{-\theta} \cap \fg^{-\sigma}\right )  
\]
 via the embedding \eqref{eq-lie-alg-2n-embedding}.
 
 When $\alpha\beta = 0$ it is clear that 
\begin{equation}
\label{eq-alpha-beta-zero}
\fg_{[\alpha{:}\beta]}^\sigma
 \cong 
 \left (\fg^{\theta} \cap \fg ^{\sigma}\right ) \ltimes  \left ( \fg^{-\theta} \cap \fg^{-\sigma}\right ) .
 \end{equation} 
Putting everything together, we have proved Theorem~\ref{thm-real-fmly-algebras1}.
 
The theorem can be re-expressed as follows.  Suppose we start with a finite-dimensional complex Lie algebra $\fg$ and two commuting conjugate-linear involutions $\sigma_1$, $\sigma_2$ of $\fg$.  
The formulas
\[
 \begin{aligned}
\theta  & : = \sigma_1\sigma_2  \\
 \sigma &:  =  \sigma_1 \\
\end{aligned}
 \]
associate to $\sigma_1$ and $\sigma_2$  two   commuting   involutions $\theta$ and $\sigma$ of $\fg$, the first complex-linear and the second conjugate-linear.    We obtain the following result:
  
\begin{theorem}
Let $\fg$ be a finite-dimensional complex   Lie algebra and let $\sigma_1$ and $\sigma_2$ be commuting conjugate linear involutions of $\fg$.   There is an  algebraic family of complex Lie algebras  $\g$ over $\C\P^1$ and a real structure on $\g$ associated with the standard real structure on $\C\P^1$ that has fixed fibers
\begin{eqnarray}\nonumber
&&\g^{\boldsymbol{\sigma}}|_{[\alpha:\beta]}\simeq \begin{cases}
\fg^{\sigma_1}  & \alpha\beta> 0\\
 \left ( \fg^{\sigma_1}\cap \fg^{\sigma_2}\right )  \ltimes \left ( \fg^{\sigma_1}\cap\fg^{-\sigma_2}\right )  &  \alpha\beta=0 \\
{\fg}^{\sigma_2}  & \alpha\beta< 0 .
\end{cases}
\end{eqnarray}
\end{theorem}

\begin{example}
For $\fg=\mathfrak{sl}(n,\C)$,  $\theta(A)=-A^t$  and $\sigma(A)=\overline{A}$, the above construction leads to the  family of real Lie algebras over $\R\P^1$  with fibers
\[
\fg_{[\alpha{:}\beta]}^\sigma
= 
\left\{ \,  
\begin{bmatrix}
X &  \alpha Y \\
  \beta Y&X
\end{bmatrix} \in \mathfrak{gl}_{2n}(\C) 
\,:\,  
X\in\mathfrak{so}(n,\R),Y\in \mathfrak{gl}(n,\R), Y^t=Y ) 
\,\right\} .
\]
It  links the split real form $\mathfrak{sl}(n,\R)$ with the compact real form $\mathfrak{su}(n)$. For $\alpha\beta\neq 0$ the map 
\[ \begin{bmatrix}
X &  \alpha Y \\
  \beta Y&X
\end{bmatrix}\longmapsto X+i^{1-\operatorname{sgn}(\alpha\beta)}Y\]
defines an isomorphism from $\fg_{[\alpha{:}\beta]}^\sigma$ onto $\mathfrak{sl}(n,\R)$ if $\alpha\beta>0$, and onto  $\mathfrak{su}(n)$ if  $\alpha\beta<0$.
\end{example} 
\section{A family of complex algebraic groups }
\label{sec-group-families}

In this section we shall construct an algebraic family of groups corresponding to the algebraic family of Lie algebras in Section~\ref{sec-algebra-families}.

We begin with a complex affine algebraic group $G\subseteq SL(n,\C)$ and an algebraic involution
\[
\theta \colon G \longrightarrow G .
\]
The formula 
\begin{equation}
\label{eq-embdg-fmla2}
g \longmapsto \frac 12 
\begin{bmatrix} 
g +  \theta (g) & g-  \theta (g) \\
g -  \theta (g) &g+  \theta (g) 
\end{bmatrix}
\end{equation}
defines a closed algebraic embedding of $G$ into $SL({2n},\C)$.
This is because 
\[
\frac 1 2 \begin{bmatrix} 
g +  \theta (g) & g-  \theta (g) \\
g -  \theta (g) &g+  \theta (g) 
\end{bmatrix} = 
A
\begin{bmatrix} 
g   &  0 \\
0 &  \theta (g) 
\end{bmatrix} 
A^{-1}
\]
where 
\[
A =  \begin{bmatrix}  \phantom{i}1 & -1 \\ \phantom{i}1 & \phantom{-}1\end{bmatrix}.
\]
Denote by $G_{1} $  the image of the embedding \eqref{eq-embdg-fmla2}; it is  a closed algebraic subgroup of $SL(2n,\C)$.

\begin{defn}
\label{def-standard-fibers}
For $z\in \C^*$, denote by    $d(z)$   the block matrix
\[
d(z) = \begin{bmatrix} z I & 0  \\ 0& I \end{bmatrix}
\]
in $GL(2n,\C)$ and define
\[
G_z = d(z^{\frac 12}) G_1 d(z^{-\frac 12})  \subseteq SL(2n, \C) .
\]
Observe that conjugation by the matrix
 \[
 d(-z^{\frac12})d(z^{-\frac 12}) =
 \begin{bmatrix}- I & 0  \\ 0& I \end{bmatrix}
 \]
  implements the involution $\theta$ on $G_1$. It follows that  the algebraic subgroup  $G_z\subseteq SL(2n,\C)$ does not depend on the choice of square root.
\end{defn}

The next proposition, and the two others below it in this section, will be proved in Section~\ref{sec-details-on-group-family}.

\begin{proposition}
\label{prop-C-star-family}
The subset
\[
  \bigsqcup _{z\in \C^*} \{ z\} \times G_z \subset \C^* \times SL(2n,\C)
\]
is  closed subvariety of the affine variety $\C^* \times SL(2n,\C)$ and an algebraic family of groups over $\C^*$.
\end{proposition}

We shall obtain our algebraic family of groups over $\C\P^1$ by taking a closure of the family above, as follows:

\begin{defn}
Denote by $\boldsymbol G$ the Zariski closure of  the set
\[
\bigsqcup _{z\in \C^*} \{ [z{:}1] \} \times G_z
\]
 in  $\C\P^1 \times SL(2n,\C)$.
\end{defn}

In view of Proposition~\ref{prop-C-star-family} the fibers of $\BG$ over points $[z{:}1]$ with $z\ne 0$ (that is the inverse images of points under the obvious projection map) are the groups $G_z$.  

\begin{defn}
\label{def-special-fibers}
Define algebraic subgroups 
\[
G_0,G_\infty \subseteq SL(2n,\C)
\]
by the formulas
\[
G_0 = \left \{ \,
 \begin{bmatrix} 
 g & 0  \\ gX & g
 \end{bmatrix}
 \,: \, g \in G^\theta\quad\text{and}\quad X \in \fg^{-\theta}   
\,\right\} .
\]
and 
\[
G_\infty = 
\left \{ \,
 \begin{bmatrix} 
 g & gX \\ 0 & g
 \end{bmatrix}
 \,: \, g \in G^\theta\quad\text{and}\quad X \in \fg^{-\theta}   
\,\right\} .
\]
Here we identify the Lie algebra $\g$ with a subspace of $\mathfrak{sl}(2n,\C)$, and so as a vector space of $n{\times}n$ matrices, in the usual way.  
\end{defn}

\begin{proposition}
\label{prop-find-special-fibers}
The fiber of the map $\BG \to \C\P^{1}$ over any $[\alpha{:}\beta]$ is the group   $G_{\alpha/\beta}$ \textup{(}we set $\alpha/\beta=\infty$ when $\beta =0$\textup{)}.\end{proposition}

 \begin{proposition} 
 \label{prop-alg-family-of-groups}
 The variety $\boldsymbol G$ is nonsingular, and is an algebraic family of groups over $\C\P^1$.
\end{proposition}

 \begin{remark}
To $\BG$ there is an associated algebraic family of Lie algebras   (for details see  \cite[sec. 2.2.1]{Ber2016}), which is  precisely the family and involution constructed in Section~\ref{sec-algebra-families}.
\end{remark}

\begin{remark}
The involution of $\C\P^1 {\times} SL(2n,\C)$ which is induced by conjugation by  $ \left [\begin{smallmatrix}- I & 0  \\ 0& I \end{smallmatrix}\right ] \in GL(2n,\C)$ restricts to an involution $\boldsymbol \theta$  of $\BG$.
It induces the involution of $\g$ that was discussed at the end of  Section~\ref{sec-algebra-families}. On each fiber of $\BG$,  the effect of $\boldsymbol \theta$  is to multiply  by $-1$ the off-block-diagonal elements. 
The fixed point subfamily $\BG^{\boldsymbol \theta}$ is the constant family over $\C\P^1$  with fiber $G^{\theta}$ (embedded block  diagonally inside $SL(2n,\C)$), and on  $\BG|_{[1:1]}=G_1$  the involution coincides with $\theta$. 
\end{remark}

 \begin{remark}
As was the case with the  family of Lie algebras, the family $\BG$ is independent (up to isomorphism) of the choice of the closed embedding of $G$ into $ SL(n,\C)$ used in the construction.  To see this, suppose given two families $\BG_1$ and $\BG_2$, as above, associated to two embeddings of $G$ into $SL(n_1,\C)$ and $SL(n_2,\C)$. Consider the diagonal embedding 
\[
G \longrightarrow   SL(n_1,\C){\times}SL(n_2,\C) \longrightarrow SL(n_1{+}n_2, \C).
\]
  Applying our construction to this third embedding,  we obtain a third  family,  $\BG_3$.
There are obvious projection morphisms from $\BG_3$ to $\BG_1$ and $\BG_2$, and these are isomorphisms.
\end{remark}

\section{Real structure on the family of   groups}
 \label{sec-group-real-structure}
 
Let $G$ be a  complex  affine algebraic group   and let $\theta$ be  an algebraic involution of $G$.  In addition,  let $\sigma$ be an antiholomorphic involution of $G$ that commutes with $\theta$.   The purpose of this section is to associate to $\sigma$ a real structure on the family $\BG$ that was constructed in the previous section (that is,  is an involution $\BG \to \overline{\BG}$ that is a morphism of families; see \cite[section 2.5.3]{Ber2016}), and then study the associated family of real forms  We shall prove the following result:

 \begin{theorem}
 \label{thm-real-fmly-groups1}
Let $G$ be a  complex   affine algebraic group, let $\theta$ be an algebraic  involution of $G$, and let $\BG$ be the algebraic family of Lie algebras over $\C\P^1$ from the previous section associated to $G$ and $\theta$.  If $\sigma $ is any antiholomorphic involution of $G$ that commutes with $\theta$, then there is an associated  real structure on $\BG$  whose  fixed fibers over the fixed space $\R\P^1\subseteq \C\P^1$ have the following isomorphism types:
\begin{eqnarray}\nonumber
&&\BG ^{\sigma}|_{[\alpha:\beta]}\cong 
 \begin{cases}
G^{\sigma} & \alpha\beta> 0\\
 \left (G^{\sigma}\cap G^{\theta}\right ) \ltimes \left ( \fg^{\sigma}\cap\fg^{-\theta}\right )&  \alpha\beta=0 \\
 {G}^{\sigma\theta }  & \alpha\beta< 0 .
\end{cases}
\end{eqnarray}
\end{theorem}

As in the Lie algebra case, for the proof of the theorem  it will   be convenient to choose a particular   embedding of $G$  into $SL(n,\C)$.  The following lemma is proved in exactly the same way as Lemma~\ref{lem-can-form-comm-inv1}.

\begin{lemma}\label{lem-can-form-comm-inv2}
There exists a closed embedding  of $\iota\colon G\to SL(n,\C)$ for some $n$, and a matrix $S\in GL(n,\R)$ such that $ S^2 = I$, and such that the antiholomorphic involution $\tau$ of $SL(n,\C)$ defined by 
\[
\tau \colon X \longmapsto  S\overline{X}S 
\]
  satisfies 
$
\iota(\sigma(g)) =  \tau( \iota(g))
$
  for every $g \in G$. \qed\end{lemma}
  
 \begin{remark}
 As in Lemma~\ref{lem-can-form-comm-inv1}, we can choose $S$ so that the fixed subgroup of the involution on $SL(n,\C)$ is $SU(k,k)$.  Thus every real form embeds into $SU(k,k)$, for some $k$.
 \end{remark}
  
\begin{proof}[Proof of Theorem~\ref{thm-real-fmly-groups1}] 
Embed $G$ into $SL(n,\C)$ as in  Lemma \ref{lem-can-form-comm-inv2} above, and let $S$ be the matrix in that lemma.   
Then embed $G$ into $SL(2n,\C)$ using the   formula 
\begin{equation*}
g \longmapsto \frac 12 
\begin{bmatrix} 
g +  \theta (g) & g-  \theta (g) \\
g -  \theta (g) &g+  \theta (g) 
\end{bmatrix}
\end{equation*}
that we used in the previous section, and apply the construction of the previous section to obtain an  algebraic family of groups $\BG$ that is embedded as  a subfamily of the constant family over $\C\P^1$ with  fiber $SL(2n,\C)$.

When it is combined with  the standard real structure on $\C\P^1$, the formula 
\[
\begin{bmatrix}
A &B \\ C & D
\end{bmatrix}
\longmapsto 
\begin{bmatrix}
S & 0 \\ 0 & S 
\end{bmatrix}
\begin{bmatrix}
\overline A & \overline  B \\ \overline  C & \overline  D
\end{bmatrix}
\begin{bmatrix}
S  & 0 \\ 0 & S 
\end{bmatrix}
\]
determines  a  real structure on the constant family over $\C\P^1$    with
  fiber $SL(2n,\C)$ . It  restricts to a real structure on $\BG$ that is associated to the standard real structure on $\C\P^1$.
  
  Let us now determine the  fixed fibers in $\BG$ over  $  \R\P^1\subseteq \C\P^1$.
  The  fibers of $\BG$   are the groups $G_z$ from Definitions~\ref{def-standard-fibers} and \ref{def-special-fibers}.  As in the Lie algebra case, we shall consider separately the cases $z>0$, $z<0$ and $z=0, \infty$.

If $z>0$, and if $\gamma\in \R$ is a square root of $z$, then  
\begin{equation}
\label{eq-pos-group}
G_z = 
\left \{ \,
\frac 12 \begin{bmatrix}
g + \theta(g) & \gamma (g - \theta(g)) \\
\gamma^{-1} (g - \theta(g)) & g + \theta(g)
\end{bmatrix}
\,:\, g \in G
\,\right \} .
\end{equation}
The map that associates to $g\in G$ the corresponding matrix displayed in \eqref{eq-pos-group} is an isomorphism of algebraic groups from $G$ to $G_z$. The fixed group  $G_z^\sigma$ consists of those matrices   for which
\[
\begin{aligned}
\sigma (g) + \sigma (\theta(g)) & = g + \theta(g) \\
\sigma (g) - \sigma(\theta(g)) & =  g  -  \theta(g)  ,
\end{aligned}
\]
and solving the equations we find that $\sigma(g) = g$.  So $G_z^\sigma \cong G^\sigma$.

If $z< 0$,  and if $\gamma\in \R$ is a square root of $-z$, then
\begin{equation}
\label{eq-neg-group}
G_z = 
\left \{ \,
\frac 12 \begin{bmatrix}
g + \theta(g) &  i\gamma (g - \theta(g)) \\
-i \gamma^{-1} (g - \theta(g)) & g + \theta(g)
\end{bmatrix}
\,:\, g \in G
\,\right \} . 
\end{equation}
The fixed group  $G_z^\sigma$ in this case consists of those matrices   for which
\[
\begin{aligned}
\sigma (g) + \sigma (\theta(g)) & = g + \theta(g) \\
\sigma (g) - \sigma(\theta(g)) & =  -(g  -  \theta(g)) ,
\end{aligned}
\]
and solving these equations we find that $\sigma(g) = \theta(g)$, or in other words $\sigma(\theta(g)) = g$.  So $G_z^\sigma \cong G^{\sigma\theta}$.

The groups $G_0$ and $G_\infty$ are given explicitly in Definition~\ref{def-special-fibers}, and it is clear that 
\[
G_0^\sigma \cong G_\infty ^\sigma \cong \left (G^\sigma \cap G^\theta\right ) \ltimes  \left (  \fg^{\sigma} \cap  \fg^{-\theta}    \right )  ,
\]
as required.\end{proof}
Theorem ~\ref{thm-real-fmly-groups1} leads to the following result:

\begin{theorem}
\label{thm-two-real-forms}
Let $G$ be a  complex   affine algebraic group and let $\sigma_1$ and $\sigma_2$ be commuting antiholomorphic involutions of $G$.   There is an  algebraic family of complex algebraic groups $\BG$ over $\C\P^1$ and a real structure on $\BG$ associated with the standard real structure on $\C\P^1$ that has fixed fibers
\begin{eqnarray}\nonumber
&&\BG^{\boldsymbol{\sigma}}|_{[\alpha:\beta]}\simeq \begin{cases}
G^{\sigma_1}  & \alpha\beta> 0\\
 \left ( G^{\sigma_1}\cap G^{\sigma_2}\right )  \ltimes \left ( \fg^{\sigma_1}\cap\fg^{-\sigma_2}\right )  &  \alpha\beta=0 \\
{G}^{\sigma_2}  & \alpha\beta< 0 .
\end{cases}
\end{eqnarray}
\end{theorem}

\begin{example}
Starting with $G=GL(n,\C)$,  set  
\[
\theta(A)=J_\theta AJ_\theta \quad \text{and} \quad \sigma(A)=J_\sigma A^{*\,-1} J_\sigma, 
\]
where  
\[
 J_\theta=\operatorname{diag}( I_{p},I_d, I_q)
 \quad \text{and} \quad 
 J_\sigma=\operatorname{diag}(  I_{p{+}d},-I_{q}) ,
\]
and where $p{+}q{+}d=n$. The corresponding family of real groups satisfies
\[
\BG^{{\boldsymbol{\sigma}}}|_{[\alpha:\beta]} \cong 
\begin{cases}
 U(p{+}d,q), & \alpha \beta> 0\\
 U(p,d{+}q), & \alpha \beta< 0 ,
\end{cases}
\]
and $\BG^{{\boldsymbol{\sigma}}}|_{[1:0]}\cong \BG^{{\boldsymbol{\sigma}}}|_{[0:1]}$ is isomorphic to the semi-direct  product of the group $U(p,q){\times}U(d)$ with the additive abelian group of $ (p{+}q){\times}d$ complex matrices, where the action of $U(p,q)$ is through ordinary matrix multiplication on the left, and the action of $U(d)$ is through matrix multiplication by the adjoint on the right.
\end{example}

\section{More on the algebraic family of groups}
\label{sec-details-on-group-family}

In this section we shall verify that our construction does indeed produce an algebraic family of groups over $\C\P^1$ \cite[Section 2.2]{Ber2016}, or in other words that produces a smooth group scheme.   The main idea is to carry out computations in the analytic topology, and then argue that the desired results in the Zariski topology follow from general principles.

\begin{lemma}
\label{lem-analytic-closure}
The subset 
\[
\bigsqcup_{[\alpha{:}\beta]\in \C\P^1} \{ [\alpha{:}\beta]\} \times G_{\alpha/\beta} \subseteq   
\C\P^1\times SL(2n, \C)
\]
 is closed in the analytic topology.  In addition it is the closure in $ \C\P^1\times SL(2n, \C)$ of its intersection with   $\C^{*}\times SL(2n, \C)$.
\end{lemma}

\begin{proof}
The intersection of the displayed subset with $\C^{*}\times SL(2n, \C)$ is certainly closed in $\C^{*}\times SL(2n, \C)$.  So we just need to show that the closure in $\C\P^1\times SL(2n, \C)$ includes in addition   the fibers  $\{[0{:}1]\} \times G_{0} $ and $\{[1{:}0]\} \times G_{\infty} $, and nothing more.  

If $g\in G^\theta$ and $X\in \fg^{-\theta}$, then the element
\[
 \frac 12
\begin{bmatrix}
g\exp(zX)+ g \exp(-zX) &  z (g\exp(zX)- g \exp(-zX)) \\
z^{-1}(g\exp(zX)- g \exp(-zX)) & g\exp(zX)+ g \exp(-zX)
\end{bmatrix}
\]
in the fiber over $[z^2{:}1]$ converges to 
$\left [ 
\begin{smallmatrix}
 g & 0 \\ gX & g
 \end{smallmatrix}
\right ]$
 as $z\to 0$. So the fiber   $\{[0{:}1]\} \times G_{0} $  is included in the closure, and similarly so is the fiber at $\infty$.
 
We shall now show that the closure includes nothing more over the point $[0{:}1]$ (the case of the point $[1{:}0]$ is similar and will be omitted).  Suppose that  $z_n\to 0$ and that
\begin{equation}
\label{eq-limit-fmla}
\frac 12 
\lim_{n\to \infty}
\begin{bmatrix} 
g_n + \theta(g_n) & z_n( g_n - \theta(g_n)) \\
z_n^{-1}( g_n - \theta(g_n)) & g_n + \theta(g_n) 
\end{bmatrix} 
=
\begin{bmatrix}
g & h \\ k & g
\end{bmatrix} .
\end{equation}
 Writing 
\[
g_n - \theta(g_n) = z_n \cdot z_n^{-1} (g_n - \theta(g_n))
\]
we find that $g_n - \theta(g_n) \to 0$.  As a result, 
$
  \lim_{n\to \infty} g_n  = g$ and 
$ \theta(g) = g $. In addition 
\[
h = \lim_{n\to \infty} z_n^2 \cdot z_n^{-1} (g_n - \theta(g_n)) =0.
\]
So it  remains to show that  the matrix  $k$ in \eqref{eq-limit-fmla} has the form $gW$ for some $W\in \fg^{-\theta}$. 

If $n$ is sufficiently large, then we can write 
\[
g _n = g \exp (X_n)
\]
where $X_n \to 0$.  By elementary calculus, there is a constant $C_1> 0$ such that  
\begin{equation}
\label{eq-exp-estimate}
\| \exp(X) - (I+X)\| \le C_1 \| X\|^2
\end{equation}
for all $X$ sufficiently close to $0$.   We obtain from this  that 
\[
 \| X -\theta (X) \| <  C_2\| \exp(X) - \exp(\theta(X)) \|
\]
for  some $C_2>0$ and all $X$ sufficiently close to $0$,
and from this we get  that
\begin{equation}
\label{eq-an-estimate}
\|  \frac 1{2z_n} (X_n -\theta (X_n)) \| < C_3 \|\frac 1{2z_n} ( g_n - \theta (g_n)) \| 
\end{equation}
for some $C_3>0$ and all $n$ sufficiently large.  On the right-hand side of \eqref{eq-an-estimate} appear the norms of a convergent sequence of matrices. So the terms $(2z_n)^{-1} (X_n - \theta (X_n))$ that appear on the left-hand side constitute a bounded sequence, and passing to a subsequence we may assume this sequence converges, say
\begin{equation}
\label{eq-limit-is-w}
\lim_{n\to \infty}    (2z_n)^{-1} (X_n - \theta (X_n)) = W.
\end{equation}
Of course $W\in \fg^{-\theta}$.

It follows from the Campbell-Baker-Hausdorff formula that there is a constant $C_3 > 0$ such that if $Y$ and $Z$ are any complex $n{\times}n$ matrices sufficiently close to $0$, then
\begin{equation}
\label{eq-cbh}
\| \exp (Y+Z) -  \exp (Y)\exp(Z)\|  \le C_4 \|Y\|\cdot \| Z\| .
\end{equation}
Now write 
\[
Y_n   =\tfrac 12 (X_n + \theta (X_n)) \quad \text{and} \quad 
 Z_n  = \tfrac 12 (X_n -\theta (X_n))   .
 \]
 Applying \eqref{eq-cbh} we get 
 \begin{multline*}
 \|
 g\bigl (\exp(X_n) -  \exp(\theta(X_n))\bigr ) - 
 g\exp(Y_n)\bigl (\exp(Z_n) -  \exp(-Z_n)\bigr ) 
 \|
  \\
 \le 
 C_4 \| Y_n\| \|Z_n\|
 \end{multline*}
 for some $C_4>0$ and all large $n$.  Multiplying by $(2z_n)^{-1}$ and applying \eqref{eq-an-estimate} we see that 
 \begin{multline*}
\lim_{n\to \infty} (2z_n)^{-1}  g\bigl (\exp(X_n) -  \exp(\theta(X_n))\bigr ) 
\\
= 
\lim_{n\to \infty}   (2z_n)^{-1} g\exp(Y_n)\bigl (\exp(Z_n) -  \exp(-Z_n)\bigr ) .
 \end{multline*}
Remembering the definition of $Z_n$, we see that the  second limit is $gW$ by \eqref{eq-exp-estimate} and \eqref{eq-limit-is-w}.
\end{proof}

\begin{lemma}
\label{lem-smooth}
The set 
\[
\bigsqcup_{[\alpha{:}\beta]\in \C\P^1} \{ [\alpha{:}\beta]\} \times G_{\alpha/\beta} \subseteq  \C\P^1\times SL(2n, \C)
\]
 is a complex submanifold, and the projection from it to $\C\P^1$ is a submersion.
\end{lemma}

\begin{proof}
This follows from the fact that the Lie algebras of the $G_z$ form a complex analytic subbundle of the trivial bundle with fiber $\mathfrak{sl}(2n,\C)$, and from the elementary theory of the exponential map.
\end{proof}

\begin{proposition}[{\cite[Proposition 7, p.12]{Serre55}}]
\label{prop-serre}
The analytic closure of the image of a regular map between complex affine varieties is Zariski closed. \qed
\end{proposition}

\begin{proposition}[{\cite[p.\ 13]{Milnor68}}]
\label{prop-milnor}
If a complex affine variety in $\C^n$ is a complex submanifold of $\C^n$, then it is nonsingular as a variety. \qed
\end{proposition}

\begin{proof}[Proof of Proposition~\ref{prop-C-star-family}]
Consider the regular map
\[
\C^{*} \times G_1 \longrightarrow \C^* \times SL (2n, \C)
\]
defined by 
\[
(w, g) \longmapsto \bigl (w^2, d(w) g d(w)^{-1} \bigr)
\]
The image is
\begin{equation}
\label{eq-subset-that-is-closed}
\bigsqcup _{z\in \C^*} \{ [z{:}1] \} \times G_z .
\end{equation}
This is closed in the analytic topology, and therefore also in the Zariski topology, by Proposition~\ref{prop-serre}.  The image  is smooth (indeed complex)   submanifold, and therefore also a nonsingular subvariety by Proposition~\ref{prop-milnor}.  The projection to $\C^*$ is a submersion of smooth manifolds, and therefore also a smooth morphism (see for example \cite[Proposition 10.4]{Hartshorne1977}).
\end{proof}

\begin{proof}[Proof of Proposition~\ref{prop-find-special-fibers}]
Recall that $\BG$ is the Zariski closure of the set \eqref{eq-subset-that-is-closed} in $\C\P^1{\times}SL(2n,\C)$. 
By Proposition~\ref{prop-C-star-family}, whose proof we just completed,  the set \eqref{eq-subset-that-is-closed} is already closed in   $\C^* \times SL (2n, \C)$, so when $\alpha/\beta \ne 0,\infty$ the fiber of $\BG$ over $[\alpha{:}\beta]$ is $G_{\alpha/\beta}$, as required.  By Proposition~\ref{prop-serre} the Zariski closure is the analytic closure, so the computation of the fibers over the remaining two points of $\C\P^1$ is handled by Lemma~\ref{lem-analytic-closure}.\end{proof}

\begin{proof}[Proof of Proposition~\ref{prop-alg-family-of-groups}]
According to Lemma~\ref{lem-smooth},  $\BG$ is a smooth submanifold of  $\C\P^{1} \times SL (2n, \C)$, and so by Proposition~\ref{prop-milnor} it is a nonsingular subvariety.  The projection to $\C\P^1$ is once again a submersion of smooth manifolds, and therefore also a smooth morphism.
\end{proof}

\bibliography{references}

\def\cprime{$'$}
\begin{thebibliography}{ABV92}

\bibitem[ABV92]{ABV}
Jeffrey Adams, Dan Barbasch, and David~A. Vogan, Jr.
\newblock {\em The {L}anglands classification and irreducible characters for
  real reductive groups}, volume 104 of {\em Progress in Mathematics}.
\newblock Birkh\"auser Boston, Inc., Boston, MA, 1992.

\bibitem[Afg15]{Afgoustidis15}
Alexandre Afgoustidis.
\newblock How tempered representations of a semisimple {L}ie group contract to
  its {C}artan motion group.
\newblock Preprint, 2015.
\newblock \href{https://arxiv.org/abs/1510.02650}{arXiv:1510.02650}.

\bibitem[AG08]{AizenbudGourevitch08}
Avraham Aizenbud and Dmitry Gourevitch.
\newblock Schwartz functions on {N}ash manifolds.
\newblock {\em Int. Math. Res. Not. IMRN}, (5):Art. ID rnm 155, 37, 2008.

\bibitem[Ber14]{Bernstein14}
Joseph Bernstein.
\newblock Stacks in representation theory. {W}hat is a continuous
  representation of an algebraic group?
\newblock Preprint, 2014.
\newblock \href{https://arxiv.org/abs/1410.0435}{arXiv:1410.0435}.

\bibitem[BHS16]{Ber2016}
Joseph Bernstein, Nigel Higson, and Eyal~M. Subag.
\newblock Algebraic families of {H}arish-{C}handra pairs.
\newblock Preprint, 2016.
\newblock \href{https://arxiv.org/abs/1610.03435}{arXiv:1610.03435}.

\bibitem[BHS17]{Ber2017}
Joseph Bernstein, Nigel Higson, and Eyal~M. Subag.
\newblock Contractions of representations and algebraic families of
  {H}arish-{C}handra modules.
\newblock Preprint, 2017.
\newblock \href{https://arxiv.org/abs/1703.04028}{arXiv:1703.04028}.

\bibitem[BK14]{BernsteinKroetz2014}
Joseph Bernstein and Bernhard Kr{\"o}tz.
\newblock Smooth {F}r\'echet globalizations of {H}arish-{C}handra modules.
\newblock {\em Israel J. Math.}, 199(1):45--111, 2014.

\bibitem[DR83]{Dooley83}
Anthony~H. Dooley and John~W. Rice.
\newblock Contractions of rotation groups and their representations.
\newblock {\em Math. Proc. Cambridge Philos. Soc.}, 94(3):509--517, 1983.

\bibitem[DR85]{Dooley85}
Anthony~H. Dooley and John~W. Rice.
\newblock On contractions of semisimple {L}ie groups.
\newblock {\em Trans. Amer. Math. Soc.}, 289(1):185--202, 1985.

\bibitem[FJ78]{Flensted-Jensen}
Mogens Flensted-Jensen.
\newblock Spherical functions of a real semisimple {L}ie group. {A} method of
  reduction to the complex case.
\newblock {\em J. Funct. Anal.}, 30(1):106--146, 1978.

\bibitem[Har77]{Hartshorne1977}
Robin Hartshorne.
\newblock {\em Algebraic geometry}.
\newblock Springer-Verlag, New York-Heidelberg, 1977.
\newblock Graduate Texts in Mathematics, No. 52.

\bibitem[Hel88]{Helm}
Aloysius~G. Helminck.
\newblock Algebraic groups with a commuting pair of involutions and semisimple
  symmetric spaces.
\newblock {\em Adv. in Math.}, 71(1):21--91, 1988.

\bibitem[Hel01]{Hel}
Sigurdur Helgason.
\newblock {\em Differential geometry, {L}ie groups, and symmetric spaces},
  volume~34 of {\em Graduate Studies in Mathematics}.
\newblock American Mathematical Society, Providence, RI, 2001.
\newblock Corrected reprint of the 1978 original.

\bibitem[Hig08]{Higson08}
Nigel Higson.
\newblock The {M}ackey analogy and {$K$}-theory.
\newblock In {\em Group representations, ergodic theory, and mathematical
  physics: a tribute to {G}eorge {W}. {M}ackey}, volume 449 of {\em Contemp.
  Math.}, pages 149--172. Amer. Math. Soc., Providence, RI, 2008.

\bibitem[Hig11]{Higson11}
Nigel Higson.
\newblock On the analogy between complex semisimple groups and their {C}artan
  motion groups.
\newblock In {\em Noncommutative geometry and global analysis}, volume 546 of
  {\em Contemp. Math.}, pages 137--170. Amer. Math. Soc., Providence, RI, 2011.

\bibitem[HS01]{HelmSchw}
Aloysius~G. Helminck and Gerald~W. Schwarz.
\newblock Orbits and invariants associated with a pair of commuting
  involutions.
\newblock {\em Duke Math. J.}, 106(2):237--279, 2001.

\bibitem[IW53]{Inonu-Wigner53}
Erdal Inonu and Eugene~P. Wigner.
\newblock On the contraction of groups and their representations.
\newblock {\em Proc. Nat. Acad. Sci. U. S. A.}, 39:510--524, 1953.

\bibitem[Loo69]{Loos}
Ottmar Loos.
\newblock {\em Symmetric spaces, I, II}.
\newblock W. A. Benjamin, 1969.

\bibitem[Mac75]{Mackey75}
George~W. Mackey.
\newblock On the analogy between semisimple {L}ie groups and certain related
  semi-direct product groups.
\newblock In {\em Lie groups and their representations ({P}roc. {S}ummer
  {S}chool, {B}olyai {J}\'anos {M}ath. {S}oc., {B}udapest, 1971)}, pages
  339--363. Halsted, New York, 1975.

\bibitem[Mil68]{Milnor68}
John Milnor.
\newblock {\em Singular points of complex hypersurfaces}.
\newblock Annals of Mathematics Studies, No. 61. Princeton University Press,
  Princeton, N.J.; University of Tokyo Press, Tokyo, 1968.

\bibitem[OV90]{OnishchikVinberg}
Arkadii~L. Onishchik and {\`Ernest}.~B. Vinberg.
\newblock {\em Lie groups and algebraic groups}.
\newblock Springer Series in Soviet Mathematics. Springer-Verlag, Berlin, 1990.
\newblock Translated from the Russian and with a preface by D. A. Leites.

\bibitem[Pan13]{Panyushev}
Dmitri~I. Panyushev.
\newblock Commuting involutions and degenerations of isotropy representations.
\newblock {\em Transform. Groups}, 18(2):507--537, 2013.

\bibitem[Ser56]{Serre55}
Jean-Pierre Serre.
\newblock G\'eom\'etrie alg\'ebrique et g\'eom\'etrie analytique.
\newblock {\em Ann. Inst. Fourier, Grenoble}, 6:1--42, 1955--1956.

\bibitem[Sub17]{Subag17}
Eyal~M. Subag.
\newblock The algebraic {M}ackey-{H}igson bijections.
\newblock Preprint, 2017.
\newblock \href{https://arxiv.org/abs/1706.05616}{arXiv:1706.05616}.

\bibitem[Vog82]{VoganDuality}
David~A. Vogan, Jr.
\newblock Irreducible characters of semisimple {L}ie groups. {IV}.
  {C}haracter-multiplicity duality.
\newblock {\em Duke Math. J.}, 49(4):943--1073, 1982.

\bibitem[Vog93]{Vogan93}
David~A. Vogan, Jr.
\newblock The local {L}anglands conjecture.
\newblock In {\em Representation theory of groups and algebras}, volume 145 of
  {\em Contemp. Math.}, pages 305--379. Amer. Math. Soc., Providence, RI, 1993.

\end{thebibliography}
\bibliographystyle{alpha}
 
\end{document}